\documentclass[12pt,reqno]{amsart}

\usepackage{graphicx,subfigure}
\usepackage{hyperref}
\hypersetup{colorlinks=true, citecolor=blue, linkcolor=red}

\numberwithin{equation}{section} \numberwithin{figure}{section}
\numberwithin{table}{section} 
{ \theoremstyle{plain}

}

{ \theoremstyle{definition}
\newtheorem{thm}{Theorem}

\newtheorem{pro}{Proposition}
\newtheorem{rem}{Remark}

\numberwithin{equation}{section} \numberwithin{lem}{section}
\numberwithin{thm}{section} \numberwithin{cor}{section}
\numberwithin{pro}{section} \numberwithin{rem}{section}
}

\begin{document}

\title[Lower bounds for the vectorial Allen-Cahn energy]{Optimal energy growth lower bounds for a class of solutions to the vectorial Allen-Cahn equation}



\author{Christos Sourdis} \address{Department of Mathematics and Applied Mathematics, University of
Crete.}
              \email{csourdis@tem.uoc.gr}           




\maketitle

\begin{abstract}
We prove optimal lower bounds for the growth of the energy over
 balls of minimizers to the vectorial Allen-Cahn energy in two
spatial dimensions, as the radius tends to infinity. In the case of
radially symmetric solutions, we can prove a stronger result in all
dimensions.
\end{abstract}


Consider the semilinear elliptic system
\begin{equation}\label{eqEq}
\Delta u=\nabla W(u)\ \ \textrm{in}\ \ \mathbb{R}^n,\ \ n\geq 1,
\end{equation}
where $W:\mathbb{R}^m\to \mathbb{R}$, $m\geq 1$, is sufficiently
smooth and nonnegative. This system has variational structure, and
solutions in a smooth bounded  domain $\Omega\subset \mathbb{R}^n$
are critical points of the energy
\[
E(v;\Omega)=\int_{\Omega}^{}\left\{\frac{1}{2}|\nabla v|^2+ W(v)
\right\}dx
\]
(subject to their own boundary conditions). A solution $u\in
C^2(\mathbb{R}^n;\mathbb{R}^m)$ is called globally minimizing if
\[
E(u;\Omega)\leq E(u+\varphi;\Omega)
\]
for every smooth bounded  domain $\Omega\subset \mathbb{R}^n$ and
for every $\varphi\in W^{1,2}_0(\Omega;\mathbb{R}^m)\cap
L^\infty(\Omega;\mathbb{R}^m)$ (see also \cite{fuscoCPAA} and the
references therein).

In the scalar case, namely $m=1$, Modica \cite{modica} used the
maximum principle to show that every bounded solution to
(\ref{eqEq}) satisfies the pointwise gradient bound
\begin{equation}\label{eqmodica}
\frac{1}{2}|\nabla u|^2\leq W(u) \ \ \textrm{in}\ \ \mathbb{R}^n,
\end{equation}
(see also \cite{cafamodica} and \cite{farinaFlat}). Using this,
together with Pohozaev identities, it was shown in \cite{modicaProc}
that the energies of such solutions satisfy the following
monotonicity property:
\begin{equation}\label{eqmonotonicity}
\frac{d}{dR}\left(\frac{1}{R^{n-1}}\int_{B(x_0,R)}^{}\left\{\frac{1}{2}|\nabla
u|^2+ W\left(u\right) \right\}dx\right)\geq 0,\ \ R>0,\ x_0\in
\mathbb{R}^n,
\end{equation}
where $B(x_0,R)$ stands for the $n$-dimensional ball of radius $R$
that is centered at  $x_0$. Combining the above two relations yields
that, if $x_0\in \mathbb{R}^n$, the ``potential'' energy of each
bounded nonconstant solution to the scalar problem satisfies the
lower bound:
\begin{equation}\label{eqenergyLowerModi}
\int_{B(x_0,R)}^{}W\left(u\right)dx\geq c R^{n-1},\ \ R>0,\ \
\textrm{for some}\ c>0.
\end{equation}
In the scalar case, the most famous representative of this class of
equations is the Allen-Cahn equation
\begin{equation}\label{eqallenSca}
\Delta u=u^3-u\ \ \textrm{in}\ \ \mathbb{R}^n, \ \ \textrm{where}\ \
W(u)=\frac{(1-u^2)^2}{4},
\end{equation}
which is used to model phase transitions (see \cite{farinaState} and
the references therein).

In the vectorial case, that is when $m\geq 2$, in the absence of the
maximum principle, it is not true in general that the gradient bound
(\ref{eqmodica}) holds (see \cite{smyrnelis} for a counterexample).
Nevertheless, it was shown by Alikakos \cite{alikakosBasicFacts}
using a stress energy tensor (see also \cite{serfaty}), and earlier
by Caffarelli and Lin \cite{caffareliLin} via Pohozaev identities,
that the energy of every solution to (\ref{eqEq}) (not necessarily
bounded) satisfies the following weak monotonicity property:
\begin{equation}\label{eqmonotoniWeak}
\frac{d}{dR}\left(\frac{1}{R^{n-2}}\int_{B(x_0,R)}^{}\left\{\frac{1}{2}|\nabla
u|^2+ W\left(u\right) \right\}dx\right)\geq 0,\ \ R>0,\ x_0\in
\mathbb{R}^n, \ n\geq 2.
\end{equation}
In fact, as was observed in the former reference, if a solution $u$
satisfies Modica's gradient bound (\ref{eqmodica}), it follows that
its energy satisfies the strong monotonicity property
(\ref{eqmodica}). Armed with (\ref{eqmonotoniWeak}), and doing some
more work in the case $n=2$ (see \cite{alikakosBasicFacts}), it is
easy to show that, if $x_0\in \mathbb{R}^n$, the energy of each
nonconstant solution to the system (\ref{eqEq}) satisfies:
\begin{equation}\label{eqGrande}
\int_{B(x_0,R)}^{}\left\{\frac{1}{2}|\nabla u|^2+ W\left(u\right)
\right\}dx\geq \left\{\begin{array}{ll}
                        c R^{n-2} & \textrm{if}\ n\geq 3,   \\
                          &     \\
                        c \ln R  & \textrm{if}\ n=2,
                      \end{array}
\right.
\end{equation}
for all $R>1$ and some $c>0$.

The above results hold for arbitrary smooth and nonnegative $W$. If
additionally $W$ vanishes at least at one point, it is easy to cook
up a suitable competitor for the energy and show that bounded
globally minimizing solutions satisfy
\[
\int_{B(x_0,R)}^{}\left\{\frac{1}{2}|\nabla u|^2+ W\left(u\right)
\right\}dx\leq CR^{n-1},\ \ R>0,\ x_0\in \mathbb{R}^n,
\]
for some $C>0$ (see for example \cite[Rem. 2.3]{ambrosioCabre}).
 The
system (\ref{eqEq}) with $W\geq 0$ vanishing at a finite number of
global minima is used to model multi-phase transitions (see
\cite{bronReih} and the references therein). In this case, the
system (\ref{eqEq}) is frequently referred to as the vectorial
Allen-Cahn equation. Under appropriate assumptions (symmetries or
non-degeneracy assumptions), it is possible to construct by
variational methods ``heteroclinic'' solutions that ``connect'' the
global minima of $W$ (see \cite{fuscoPreprint,guiSchatz,saez} and
the references therein); the energy of these solutions over
$B(x_0,R)$ is of order $R^{n-1}$ as $R\to \infty$. This observation
implies that the estimate (\ref{eqGrande}) is far from optimal for
this class of $W$'s. On the other side, for the case of the
Ginzburg-Landau system
\[
\Delta u=\left(|u|^2-1 \right)u,\ \ u:\mathbb{R}^2\to \mathbb{R}^2,
\ \ \left(\textrm{here}\ W(u)=\frac{\left(1-|u|^2\right)^2}{4}\
\textrm{vanishes on}\ \mathbb{S}^1 \right),
\]
there are globally minimizing solutions with   energy  over
$B(x_0,R)$ of order $\ln R$ as $R\to \infty$ (see
\cite{bethuel,serfaty} and the references therein). In other words,
the estimate (\ref{eqGrande}) captures the optimal growth in the
case of globally minimizing solutions to the Ginzburg-Landau system.

In this note, we will establish the corresponding optimal lower
bound in the case of the phase transition case when  $n=2$. In fact,
we will prove the analog of the lower bound
(\ref{eqenergyLowerModi}). As will be apparent, our proof does not
work if $n\geq 3$. To the best of our knowledge, there is no related
published result. Our approach combines ideas from two disciplines:
\begin{itemize}
  \item We
 adapt to this setting clearing-out arguments from the study of the
 Ginzburg-Landau system, see \cite{bethuel}.
  \item We employ variational maximum principles for globally minimizing
  solutions that have been  recently devised and  used for the study
  of the vectorial Allen-Cahn equation in \cite{alikakosPreprint}.
\end{itemize}

Our main   result is the following.

\begin{thm}\label{thmMine}
Assume that $W\in C^1(\mathbb{R}^m;\mathbb{R})$, $m\geq 1$, and that
there exist finitely many $N\geq 1$ points $a_i\in \mathbb{R}^m$
such that
\begin{equation}\label{eqpoints}
W(u)>0\ \ \textrm{in}\ \ \mathbb{R}^m\setminus \{a_1,\cdots, a_N \},
\end{equation}
and there exists small $r_0>0$ such that the functions
\begin{equation}\label{eqmonot}
r\mapsto W(a_i+r\nu),\ \ \textrm{where}\ \ \nu \in \mathbb{S}^1, \ \
\textrm{are strictly increasing\ for}\ r\in (0,r_0),\ \
i=1,\cdots,N.
\end{equation}
Moreover, we assume that
\begin{equation}\label{eqinf}\liminf_{|u|\to \infty}
W(u)>0.\end{equation}

 If $u\in C^2(\mathbb{R}^2;\mathbb{R}^m)$ is
a bounded, nonconstant, and globally minimizing solution to the
elliptic system
\begin{equation}\label{eqEq2}
\Delta u=\nabla W(u)\ \ \textrm{in}\ \ \mathbb{R}^2,
\end{equation}
for any $x_0\in \mathbb{R}^2$, there exist   constants $c_0, R_0>0$
such that
\[
\int_{B(x_0,R)}^{}W\left(u(x) \right) dx\geq c_0R\ \ \textrm{for}\ \
R\geq R_0.
\]
\end{thm}
\begin{proof}
Since the problem is translation invariant, without loss of
generality, we may carry out the proof for $x_0=0$.

 Suppose, to the
contrary, that there exists  a bounded, nonconstant, and globally
minimizing solution $u$ and radii $R_j\to \infty$ such that
\begin{equation}\label{eqcontra}
\int_{B(0,R_j)}^{}W\left(u(x) \right) dx=o(R_j)\ \ \textrm{as}\ \
j\to \infty.
\end{equation}
By the co-area formula (see for instance  \cite[Ap. C]{evans}), the
nonnegativity of $W$, and the mean value theorem, there exist
\[
s_j\in \left(\frac{R_j}{2},R_j\right)
\]
such that
\[
\int_{\partial B(0,s_j)}^{}W\left(u(x) \right) dS(x)=o(1)\ \
\textrm{as}\ \ j\to \infty.
\]

From this, as in the clearing-out lemma of \cite{bethuel}, it
follows that
\begin{equation}\label{eqbdry}
\max_{|x|=s_j}W\left(u(x) \right)=o(1)\ \ \textrm{as}\ \ j\to
\infty.
\end{equation}
Indeed, if not, passing to a subsequence if necessary, there would
exist $c_1>0$ and $x_j\in
\partial B(0,s_j)$ such that
\[
W\left(u(x_j) \right)\geq c_1\ \ \textrm{for}\ \ j\geq 1.
\]
On the other hand, since $u$ is bounded in $\mathbb{R}^2$, by
standard interior elliptic regularity estimates (see
\cite{evans,Gilbarg-Trudinger}), the same is true for $\nabla u$.
Hence, there exists $r_*>0$ such that
\[
W\left(u(x) \right)\geq \frac{c_1}{2},\ \ x\in B(x_j,r_*),\ \
\textrm{for}\ \ j\geq 1,
\]
which implies that
\[
\int_{\partial B(0,s_j)}^{}W\left(u(x) \right) dS(x)\geq
\frac{c_1}{2} \mathcal{H}^{1}\left\{B(x_j,r_*)\cap \partial B(0,s_j)
\right\}\geq c_2\ \ \textrm{for}\ \ j\geq 1.
\]
for some $c_2>0$. Clearly, the above relation contradicts
(\ref{eqcontra}).

In view of (\ref{eqinf}), relation (\ref{eqbdry}) implies that there
exist $i_j\in \{1,\cdots, N \}$ such that
\[
\max_{|x|=s_j}\left|u(x)-a_{i_j} \right|=o(1)\ \ \textrm{as}\ \ j\to
\infty.
\]
By virtue of (\ref{eqmonot}), exploiting the fact that $u$ is a
globally minimizing solution, we can apply a recent variational
maximum principle from \cite{alikakosPreprint} to deduce that
\[
\max_{|x|\leq s_j}\left|u(x)-a_{i_j} \right|\leq
\max_{|x|=s_j}\left|u(x)-a_{i_j} \right|\ \ \textrm{for}  \ \ j\gg
1,
\]
(so that the righthand side is smaller than $r_0/2$). The idea is
that, if this is violated, one can construct a suitable function
which agrees with $u$ on $\partial B(0,s_j)$ but with less energy,
thus contradicting the minimality of $u$.
 The above two
relations imply the existence of an $i_0\in \{1,\cdots,N\}$ such
that
\[
\max_{|x|\leq s_j}\left|u(x)-a_{i_0} \right|=o(1)\ \ \textrm{as}\ \
j\to \infty.
\]
Now, letting $j\to \infty$ in the above relation yields that
$u\equiv a_{i_0}$ which contradicts our assumption that $u$ is
nonconstant.
\end{proof}

\begin{rem}
Clearly, the monotonicity condition (\ref{eqmonot}) is satisfied if
the global minima are nondegenerate (the Hessian matrix $D^2W(a_i)$
is invertible for all $i=1,\cdot, N$).
\end{rem}

\begin{rem}
In dimensions $n\geq 2$, a Liouville type theorem of Fusco
\cite{fuscoCPAA} tells us that, if $W$ is as in Theorem
\ref{thmMine}, nonconstant global minimizing solutions to
(\ref{eqEq}) must enter any
neighborhood of at least two of the global minima. Intuitively, this
suggests that the optimal lower bound for the growth of the energy,
that is kinetic (or interfacial) and potential, should be of order
$R^{n-1}$ in all dimensions. In this regard, see
\cite{cafa-cordoba,modica} for the scalar case ($m=1$), provided
that the global minima are nondegenerate.
\end{rem}

If we restrict our attention to radially symmetric solutions, we
have a  stronger result which follows at once from the following
proposition which is of independent interest.

\begin{pro} Let $W\in C^{1}(\mathbb{R}^m;\mathbb{R})$, $m\geq 1$, possibly sign-changing. If $u\in
C^2(\mathbb{R}^n;\mathbb{R}^m)$, $n\geq 2$, satisfies (\ref{eqEq}),
we have that
\begin{equation}\label{eq1}
\frac{1}{2}\left|u'(R)\right|^2\leq W\left(u(R) \right)-W\left(u(0)
\right),\ \ R>0,
\end{equation}
and
\begin{equation}\label{eq2}
\frac{d}{dR}\left(\frac{1}{R^n}\int_{B(0,R)}^{}\left\{\frac{n-2}{2}|\nabla
u |^2+ n W(u)\right\} dx\right)\geq 0,\ \ R>0.
\end{equation}
\end{pro}
\begin{proof}
We know that
\[
u''+\frac{n-1}{r}u'-\nabla W(u)=0,\ \ r>0,\ \ u'(0)=0.
\]
So, letting
\[
e(r)=\frac{1}{2}\left|u'(r)\right|^2-W\left(u(r) \right),\ \ r>0,
\]
we find that
\begin{equation}\label{eq3}
e'(r)=-\frac{n-1}{r}\left|u'(r)\right|^2,\ \ r>0.
\end{equation}
Then, estimate (\ref{eq1}) follows at once by integrating the above
relation over $(0,R)$.

By Pohozaev's identity (the idea is to test the equation by $r
u'(r)$, see for instance \cite[Ch. 5]{serfaty}), for $R>0$, we have
that
\[
\frac{1}{R}\int_{B(0,R)}^{}\left\{\frac{n-2}{2}|\nabla u |^2+ n
W(u)\right\}dx= \int_{\partial B(0,R)}^{}\left\{W\left(u(R)
\right)-\frac{1}{2}\left|u'(R)\right|^2
\right\}dS=-\mathcal{H}^1\left\{\mathbb{S}^1\right\}R^{n-1}e(R).
\]
Then, we can arrive at (\ref{eq2}) by dividing both sides by
$R^{n-1}$, differentiating, and using (\ref{eq3}).
\end{proof}

\begin{rem}
 Radial solutions to the Allen-Cahn equation (\ref{eqallenSca}), decaying to zero in
an oscillatory manner, as $r\to \infty$, have been constructed in
\cite{guiRadial}.
\end{rem}

\end{document}